\documentclass{proc-l}
\usepackage{graphicx}
\copyrightinfo{2018}{American Mathematical Society}

\newtheorem{theorem}{Theorem}[section]
\newtheorem{lemma}[theorem]{Lemma}

\numberwithin{equation}{section}
\def\pput(#1,#2)#3{\noindent\smash{\raise#2pt\hbox to 0pt
   {\kern #1pt #3\hss}}\ignorespaces}

\begin{document}

\title[Approximation of \boldmath $x^n$]
{Rational approximation of \boldmath $x^n$}

\author[Nakatsukasa]{Yuji Nakatsukasa}
\address{Mathematical Institute, University of Oxford, Oxford, OX2 6GG, UK}
\email{nakatsukasa@maths.ox.ac.uk}
\author[Trefethen]{Lloyd N. Trefethen}
\address{Mathematical Institute, University of Oxford, Oxford, OX2 6GG, UK}
\email{trefethen@maths.ox.ac.uk}


\commby{}


\begin{abstract}
Let $E_{kk}^{(n)}$ denote the minimax (i.e., best supremum
norm) error in approximation of $x^n$ on $[\kern .3pt 0,1]$
by rational functions of type $(k,k)$ with $k<n$.  We show
that in an appropriate limit $E_{kk}^{(n)} \sim 2\kern .3pt
H^{k+1/2}$ independently of $n$,
where $H \approx 1/9.28903$ is Halphen's constant.
This is the same formula as for minimax approximation of \kern
.7pt $e^x$ on $(-\infty,0\kern .3pt]$.
\end{abstract}

\maketitle

\section{Introduction}

We consider minimax approximation of $x^n$ on $[\kern .3pt 0,1]$,
that is, best approximation with respect to the supremum norm
$\|\cdot \|$ on $[\kern .3pt 0,1]$.  Although $n$ is usually
thought of as an integer, we permit it to be any nonnegative
real number.  If $n$ is an even integer,
approximation of $x^n$ on 
$[-1,1]$ is equivalent to approximation
of $x^{n/2}$ on $[\kern .3pt 0, 1]$, and results will be
stated for both intervals.

For each integer $k\ge 0$, there is a unique minimax approximant
$p_k^{(n)}$ of $x^n$ among polynomials of degree at most $k$~\cite{atap}.
Let $E_k^{(n)} = \|x^n - p_k^{(n)}\|$ denote the associated error, which
will be nonzero whenever $k<n$.
In 1976 Newman and Rivlin~\cite{newman} 
published theorems showing
\begin{equation}
E_k^{(n)} \approx \textstyle{\frac{1}{2}}\kern .5pt
\hbox{erfc}(k/\sqrt{n}\kern 1.2pt),
\label{newriv}
\end{equation}
where $\hbox{erfc}(s) = 2\kern .3pt \pi^{-1/2}\int_s^\infty
\exp(-t^2)\kern .5pt dt$ is the complementary error function. 
(The constant $1/2$ is our own, based on numerical experiments.)
This formula implies
that a degree $k = O(\sqrt n\kern 1pt)$ suffices for polynomial
approximation of $x^n$ to high accuracy.  To illustrate this
effect, Figure~\ref{fig1} plots $E_k^{(n)}$ against $k^2$ for
the cases $n = 250$ and $1000$, showing good agreement with
(\ref{newriv}).  The data in our two figures have been computed
with the {\tt minimax} command in Chebfun~\cite{chebfun,minimax}.

\begin{figure}
\begin{center}
\includegraphics[scale=.8]{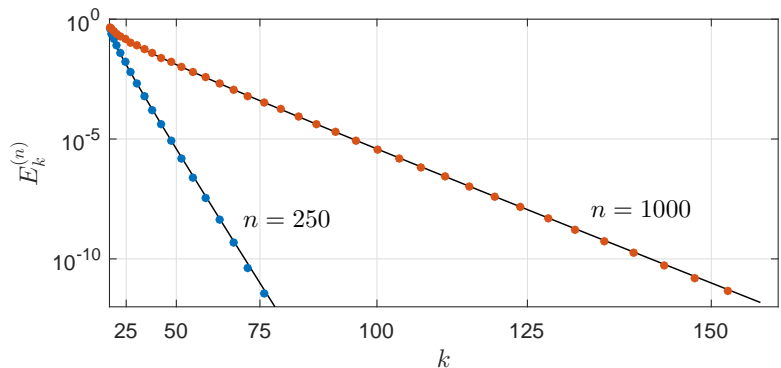}
\end{center}
\caption{\label{fig1} Errors in minimax approximation of $x^n$ by
polynomials of degree $k$; the solid lines show the
approximation~{\rm (\ref{newriv})}.  The convergence is exponential as a 
function of $k^2/n$.  The horizontal axis is scaled quadratically, so
this behavior shows up as straight lines.}
\end{figure}

\begin{figure}
\begin{center}
\includegraphics[scale=.8]{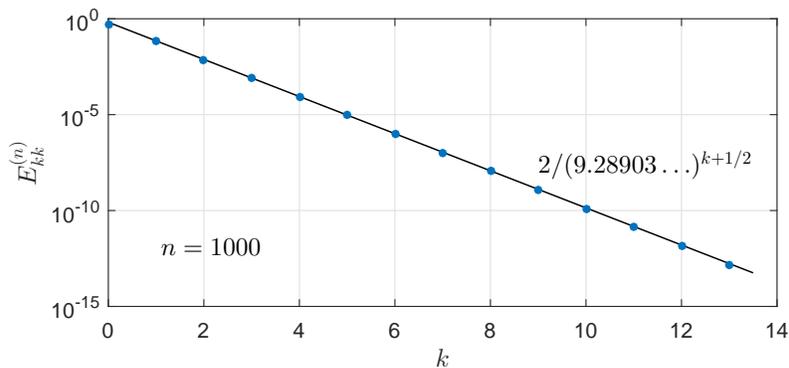}
\end{center}
\caption{\label{fig2} With rational functions of type $(k,k)$,
the convergence is much faster: exponential as a function of $k$, approximately
independently of $n$.  The solid line shows the approximation
{\rm (\ref{model})}. In this experiment $n=1000$, but the data would
be approximately the same for any other large value of $n$.}
\end{figure}

We have found that rational functions are far
more effective at approximating $x^n$ than polynomials.
To be precise, consider
approximation by real rational functions of
type $(k,k)$, that is, functions that can be written in the
form $r(x) = p(x)/q(x)$ where $p$ and $q$ are real polynomials
of degree at most $k$.  Again, standard theory shows that for
each nonnegative real number $n$ and each nonnegative integer $k$,
there exists a unique minimax approximant
$r_{kk}^{(n)}$~\cite{atap}; we denote the
error by $E_{kk}^{(n)} = \|x^n - r_{kk}^{(n)}\|$.
Here we will prove that, as illustrated in
Figure~\ref{fig2}, the errors are closely approximated by the formula
\begin{equation}
E_{kk}^{(n)} \approx 2\kern .3pt H^{k+1/2}, \quad H = 1/9.2890254919208\dots ,
\label{model}
\end{equation}
which has the remarkable property of being independent of $n$.
The number $H$, known as Halphen's constant, appears in the
problem of approximation of $\exp(x)$ for $x\in (-\infty, 0]$,
where the minimax errors are asymptotic to exactly the same
expression (\ref{model}).  Chapter 25 of~ \cite{atap} gives a
review of this famous problem of rational approximation theory,
a story that among others has involved Aptekarev, Carpenter,
Cody, Gonchar, Gutknecht, Magnus, Meinardus, Rakhmanov, Ruttan,
Trefethen, and Varga.  Halphen first identified the number now
named after him in 1886~\cite{halphen}, though not in connection
with approximation theory.

\section{Theorems}

To prove that the errors satisfy an estimate of
the form~(\ref{model}), we exploit the fact that the set of
rational functions of type $(k,k)$ is invariant under
M\"obius transformation.  In particular, we transplant
the approximation domain $[\kern .3pt 0,1]$ to $(-\infty,0\kern .3pt ]$
by the M\"obius transformation that maps
$x = 0$, $1$, and $1+1/(n-1)$ to $s = -\infty$, $0$, and $1$:
$$ x = \frac{n}{n-s}, \qquad s = \frac{n(x-1)}{x}. $$
The function $x^n$ transplants to
\begin{equation}
x^n = (n/(n-s))^n = (1-s/n)^{-n},
\end{equation}
and this establishes our first lemma.

\begin{lemma}
\label{trans}
For any real number $n>0$ and integer $k\ge 0$, the
error $E_{kk}^{(n)}$ in type $(k,k)$ minimax approximation of $x^n$
on $[\kern .3pt 0,1]$ is equal to the error in type $(k,k)$
minimax approximation of\/ $(1-s/n)^{-n}$ on $(-\infty,0\kern .3pt ]$.
\end{lemma}

Our second lemma quantifies the fact that $(1-s/n)^{-n}\approx e^s$
for $s\in (-\infty,0\kern .3pt]$.

\begin{lemma}
\label{lem}
For any $n\in (0,\infty)$ and $s\in (-\infty,0)$, 
\begin{equation}
0 < (1-s/n)^{-n} - e^s \le \frac{1}{e\kern .3pt n}. 
\end{equation}
\end{lemma}

\begin{proof}
Given $n$, define $g(s) = (1-s/n)^{-n}- e^s$,
with $g(-\infty) = g(0) = 0$.
From a binomial series we may verify
$(1-s/n)^n < e^{-s}$ for each~$s$, and taking
reciprocals establishes
$(1-s/n)^{-n} > e^{s}$, i.e., $g(s) > 0$.
The maximum value of $g(s)$ will be attained at a point
$s=\sigma$ where the derivative
\begin{displaymath}
g'(s) = (1-s/n)^{-(n+1)} - e^s
\end{displaymath}
is zero, i.e.,
$(1-\sigma/n)^{-(n+1)} = e^\sigma$.
At such a point we calculate
\begin{displaymath}
g(\sigma) = e^\sigma(1-\sigma/n)-e^\sigma = -\sigma e^\sigma\kern -1pt/n,
\end{displaymath}
and to complete the proof we note that $0 < -\sigma e^\sigma\le 1/e$ for
$\sigma \in (-\infty,0)$.
\end{proof}

We can now derive our main result.

\begin{theorem}
\label{thm1}
The errors in type $(k,k)$ rational minimax approximation
of\/ $x^n$ on $[\kern .3pt 0,1]$ satisfy
\begin{equation}
\lim_{k\to\infty}\, \lim_{n\to\infty\vphantom{k}}\, 
E_{kk}^{(n)}
\kern -3pt\left/\vrule width 0pt height 10pt depth 0pt\right.\kern -3pt
2\kern .3pt H^{k+1/2}\kern 1pt  = \kern 1pt 1, 
\label{thm1eq}
\end{equation}
where $H \approx 1/9.28903$ is Halphen's constant.
In this formula\/ $n$ may range over nonnegative real numbers or over
nonnegative integers.
\end{theorem}

\begin{proof}
Let $F_{kk}^{}$ denote the error in minimax 
type $(k,k)$ rational approximation of $e^s$ on $(-\infty,0\kern .3pt]$.
Aptekarev~\cite{aptek} established the identity
\begin{equation}
\lim_{k\to\infty}\, F_{kk}^{}
\kern -1pt\left/\vrule width 0pt height 10pt depth 0pt\right.\kern -3pt
2H^{k+1/2} \kern 1pt =\kern 1pt  1,
\label{aptekeq}
\end{equation}
which had been conjectured earlier by Magnus~\cite{magnus}.
On the other hand Lemmas~\ref{trans} and~\ref{lem} imply
\begin{equation}
F_{kk}^{} = \lim_{n\to\infty} E_{kk}^{(n)}.
\label{EandF}
\end{equation}
Equation (\ref{thm1eq}) follows from (\ref{aptekeq}) and (\ref{EandF}).
\end{proof}

Equation (\ref{thm1eq}) says little about the errors associated
with any finite value of~$n$.  Numerical data such as those
plotted in Figure~\ref{fig2} suggest that much sharper estimates
are probably valid, with $E_{kk}^{(n)}$ coming much closer to
$2\kern .3pt H^{k+1/2}$ than is shown by our arguments.

As mentioned at the outset, approximation of $x^n$ on $[\kern
.3pt 0,1]$ is equivalent to approximation of $x^{n/2}$ on $[\kern
.3pt -1,1]$ when $n$ is an even integer.  The equivalence is
spelled out in the proof of the following theorem, which uses
the same notation $E_{kk}^{(n)}$ for $[-1,1]$ as used previously
for $[\kern .3pt 0,1]$.

\begin{theorem}
\label{thm2}
The errors in type $(k,k)$ rational minimax approximation
of\/ $x^n$ on $[-1,1]$ satisfy
\begin{equation}
\lim_{k\to\infty}\,
\lim_{\vrule width 0pt height 5pt
\scriptstyle{n\to\infty}\atop{\scriptstyle{n\hbox{\scriptsize\rm~even}}}}\,
E_{kk}^{(n)} 
\kern -3pt\left/\vrule width 0pt height 10pt depth 0pt\right.\kern -3pt
2H^{\lfloor k/2\rfloor +1/2} \kern 1pt = \kern 1pt 1.
\label{thm2eq}
\end{equation}
In this formula\/ $n$ ranges over nonnegative even integers.
\end{theorem}

\begin{proof} 
Let $n$ be a nonnegative even integer.
If $k$ is even, then by the change of
variables $s = x^2$ (see for example p.~213
of~\cite{atap}), we find that type 
$(k,k)$ approximation of $x^n$ on $[-1,1]$
is equivalent to type $(k/2,k/2)$ approximation of $x^{n/2}$ on $[\kern
.3pt 0,1]$.  If $k$ is odd, then the uniqueness of best approximants
implies that the type 
$(k,k)$ approximant of $x^n$ on $[-1,1]$ must still be even, hence
the same as the type
$(k-1,k-1)$ approximant (see e.g.\ Exercise~24.1 of~\cite{atap}).
These observations justify the 
floor function $\lfloor k/2\rfloor$ of (\ref{thm2eq}).
\end{proof}

If $n$ is odd, the errors are approximately but not exactly
the same.

\section{Discussion}

In~\cite{atap} it is emphasized that rational approximants tend
to greatly outperform polynomials in cases where (i) the function
to be approximated has a nearby singularity or (ii) the domain of
approximation is unbounded.  Approximation of $x^n$ on $[\kern
.3pt 0,1]$ is essentially a problem of type (i), with nearly
singular behavior at $x\approx 1$ (not technically singular,
of course, but one could speak of a ``pseudo-singularity'').
It is interesting that the proof of Theorem~\ref{thm1} proceeds
by conversion to an equivalent problem of type (ii).

The $k = O(\sqrt n\kern 1pt )$ effect for polynomial
approximation of $x^n$ has practical consequences.  For example,
Chebfun's method of numerical computation with functions
depends on representing them adaptively to approximately
machine precision (${\approx} \kern 1pt 16$ digits) by
Chebyshev expansions.  Table~\ref{table1} lists the degrees $k$
of the Chebfun polynomials representing various powers $x^n$
on $[\kern .3pt 0,1]$.  We see that for small $n$, the system
requires $k=n$, but for larger values, each quadrupling of $n$
brings just approximately a doubling of $k$.

\begin{table}
\caption{\label{table1}Chebfun~\cite{chebfun} constructs a polynomial of an
adaptively determined degree $k$
to represent a function on a given interval
to about 16 digits of accuracy.  For
$x^n$ on $[\kern .3pt 0,1]$,
$k=n$ is needed for smaller values of $n$, whereas for larger
values, $k$ grows at a rate $O(\sqrt n\kern 1pt)$ consistent
with (\ref{newriv}).}
\begin{center}
\begin{tabular}{cccccccc}
$n~~$ & 1 & 4 & 16 & 64 & 256 & 1024 & 4096 \\[3pt]
$k~~$ & 1 & 4 & 16 & 44 & 91 & 178 & 349 \\ [5pt]
\end{tabular}
\end{center}
\end{table}

Another aspect of the $k = O(\sqrt n\kern 1pt )$ effect is
discussed by Cornelius Lanczos in a fascinating video recording
from 1972 available online~\cite{lanczos} (beginning at about
time 10:00); for a written discussion see chapter~5 of his
book~\cite{lbook}.  Lanczos speaks of the monomials $\{x^n\}$
as a ``tremendously nonorthogonal system,'' a fact quantified
by the M\"untz--Sz\'asz theorem~\cite{rudin}, and observes
that it was this effect that led him to invent what are now
called Chebyshev spectral methods for the numerical solution
of differential equations~\cite{boyd,lbook,atap}.  Numerical
analysts would rarely cite the M\"untz--Sz\'asz theorem,
but they are well aware that monomials provide exponentially
ill-conditioned bases on real intervals, making them nearly
useless for numerical computing, whereas suitably scaled
Chebyshev polynomials are excellent for computation because
they give well-conditioned bases.

These observations pertain to polynomial approximation of $x^n$,
whereas the new results of this paper concern the much greater
power of rational approximations.  There is some previous
literature on rational approximation of $x^n$, and an early
survey can be found in~\cite{reddy}.  The most developed part
of this problem has been the case in which $n$ is a fixed
positive number that is not an integer and $k\to\infty$.
Here one obtains root-exponential convergence with respect
to $k\kern .4pt$; see~\cite{stahl} for both sharp results and a survey
of earlier work.  The more basic phenomenon considered in
the present paper of exponential convergence for $k\to\infty$
for large $n$ seems not to have been noted previously, nor, in
particular, the connection with ${\approx}\kern 1pt 9.28903$.
Perhaps there may be applications where this too will have
practical consequences.

\bibliographystyle{amsplain}

\end{document}